\documentclass[11pt]{amsart}
\usepackage{graphicx}
\usepackage{amssymb}
\usepackage{amsmath}
\usepackage{amsthm,amsfonts,bbm}
\usepackage{amscd}
\usepackage{geometry}
\usepackage[all,2cell]{xy}
\usepackage{epsfig,epstopdf}
\usepackage{mathrsfs}
\usepackage{hyperref}
\usepackage{cite}
\usepackage{color}

\UseAllTwocells \SilentMatrices

\newtheorem{thm}{Theorem}[section]
\newtheorem{cor}[thm]{Corollary}
\newtheorem{lem}[thm]{Lemma}
\newtheorem{fact}[thm]{Fact}
\newtheorem{pro}[thm]{Problem}
\newtheorem{prop}[thm]{Proposition}

\theoremstyle{definition}

\theoremstyle{remark}
\newtheorem{rmk}[thm]{\bf Remark}

\newtheorem{conj}[thm]{\bf Conjecture}

\numberwithin{equation}{section}
\numberwithin{figure}{section}
\date{\today}
\begin{document}
\title[Spectral extremal problems for the  $(p,Q)$-spectral radius of hypergraphs]{Spectral extremal problems for  the $(p,Q)$-spectral radius of hypergraphs}

\author[J. Zheng]{Jian Zheng}
\address{School of  Mathematics and Statistics, Jiangxi Normal University, Nanchang 330022,  China}
\email{zhengj@jxnu.edu.cn}

\author[H. Li]{Honghai Li$^\dag$}
\address{School of  Mathematics and Statistics, Jiangxi Normal University, Nanchang 330022,  China}
\email{lhh@jxnu.edu.cn}
\thanks{$^\dag$The corresponding author. H. Li was supported by National Natural Science Foundation of China (No. 12161047) and   Jiangxi Provincial Natural Science Foundation (No. 20224BCD41001).}

\author[L. Su]{Li Su$^\ddag$}
\address{School of  Mathematics and Statistics, Jiangxi Normal University, Nanchang 330022,  China}
\email{suli@jxnu.edu.cn}
\thanks{$^\ddag$L. Su was supported by National Natural Science Foundation of China (No. 12061038).}

%\subjclass[2000]{Primary 05C65, 15A69; Secondary 13P15, 14M99}

%\newline $^{\dag}$ Corresponding author.
%\newline  E-mail addresses: zhengj@jxnu.edu.cn, lhh@jxnu.edu.cn
%\newline School of  Mathematics and Statistics, Jiangxi Normal University, Nanchang,  Jiangxi 330022,  China.}

\begin{abstract}
Let $Q$ be an $s$-vertex $r$-uniform hypergraph, and let $H$ be an $n$-vertex $r$-uniform hypergraph. Denote by $\mathcal{N}(Q,H)$  the number of isomorphic copies of $Q$ in $H$. For a hereditary family $\mathcal{P}$  of $r$-uniform hypergraphs, define
$$\pi(Q,\mathcal{P}):=\lim\limits_{n\to \infty}\binom{n}{s}^{-1}\max\{\mathcal{N}(Q,H): H\in \mathcal{P}~~\mbox{and}~~|V(H)|=n\}.$$
For $p\geq1$, the $(p,Q)$-spectral radius of $H$  is defined as $$\lambda^{(p)}(Q,H)=\max_{\|\mathbf{x}\|_{p}=1}s!\sum_{\{i_{1},\ldots,i_{s}\}\in \binom{[n]}{s}}\mathcal{N}(Q,H[\{i_{1},\ldots,i_{s}\}])x_{i_{1}}\cdots x_{i_{s}}.$$
 %generalizing the concept of the $p$-spectral radius introduced by %Keevash, Lenz, and Mubayi \cite{KLM2014}.

 In this paper, we present a systematically  investigation of the parameter $\lambda^{(p)}(Q,H)$. First, we prove that the limit
$$\lambda^{(p)}(Q,\mathcal{P}):=\lim\limits_{n\to \infty}n^{s/p-s}\max\{\lambda^{(p)}(Q,H): H\in \mathcal{P}~~\mbox{and}~~|V(H)|=n\}$$
exists, and for $p>1$, it satisfies
$$\pi(Q,\mathcal{P})=\lambda^{(p)}(Q,\mathcal{P}).$$
 Second, we study spectral generalized Tur\'an problems.
 Specifically, we establish a spectral stability result and apply it to derive a spectral version of the Erd\H{o}s Pentagon Problem: for $p\geq1$ and sufficiently large $n$, the balanced blow-up of $C_{5}$ maximizes $\lambda^{(p)}(C_{5},H)$ among all $n$-vertex triangle-free graphs $H$, thereby improving a result of Liu \cite{Liu2025}. Furthermore, we show  that for $p\geq1$ and sufficiently large $n$, the $l$-partite Tur\'an graph $T_{l}(n)$ attains the maximum $\lambda^{(p)}(K_{s},H)$ among all $n$-vertex F-free graphs $H$, where $F$ is an edge-critical graph with $\chi(F)=l+1$. This  provides a spectral analogue of a theorem due to Ma and Qiu \cite{MQ2020}.

\end{abstract}

\maketitle
\section{Introduction}
A  \emph{hypergraph} $H=(V(H),E(H))$ consists of a vertex set $V(H)=\{v_1,v_2,{\cdots},v_n\}$  and an edge set $E(H)=\{e_1,e_2,{\cdots},e_m\}$ , where $e_i \subseteq V$ for $i \in [m]:=\{1,2,\ldots,m\}$. The \emph{order} and  \emph{size} of $H$ are defined as  $\nu(H):=|V(H)|$ and  $e(H):=|E(H)|$, respectively. If $|e_i|=r$ for each $i \in [m]$ and $r \geq2$, then $H$ is called an \emph{$r$-uniform} hypergraph (or \emph{$r$}-graph).
 A simple graph is exactly a  $2$-uniform hypergraph.  The subgraph $H[I]$  is  \emph{induced} by $I$, that is, $V(H[I])=I$ and $E(H[I])=\{e\in E(H):e\subseteq I\}$. For any vertex $v\in V(H)$, we write $H-v$ for the subgraph of $H$  induced by $V(H)\backslash\{v\}$.
 For $l\geq r\geq2$, an  $r$-graph is called \emph{$l$-partite} if its vertex set can be divided into $l$ parts such that each edge has at most one vertex from each part. An edge maximal $l$-partite $r$-graph is called \emph{complete $l$-partite}. Let $T^{r}_{l}(n)$ be
the complete $l$-partite $r$-graph on $n$ vertices without two part sizes differing by more than one; when $r=2$, the graph $T^{2}_{l}(n)$ is Tur\'an graph $T_{l}(n)$.

Given an $s$-vertex $r$-graph $Q$ and an $r$-graph $H$, let $\mathcal{N}(Q,H)$ denote the number of isomorphic copies of $Q$ in $H$.
For example,  for the complete $r$-graph $K_{s}^{r}$ on $s$ vertices, we have $\mathcal{N}(Q,K_{s}^{r})=\frac{s!}{|Aut(Q)|}$.
For a family $\mathcal{F}$ of  $r$-graphs, we say a hypergraph $G$ is \emph{$\mathcal{F}$-free} if $G$ does not
contain any member of  $\mathcal{F}$ as a subgraph. The generalized Tur\'an number $ex(n,Q,\mathcal{F})$ is  the largest $\mathcal{N}(Q,H)$ among all the $n$-vertex $\mathcal{F}$-free $r$-graphs $H$. The function  $ex(n,Q,\mathcal{F})$ is a well-studied parameter; a comprehensive survey can be found in \cite{GP2025}.
Let $E(Q,H)$  denote the collection of  all $s$-subsets $I$ of $V(H)$ such that  $\mathcal{N}(Q,H[I])>0$, and define $E_{Q,H}(v)=\{I\in E(Q,H):v\in I\}$. The \emph{$Q$-degree} of $v$, denoted $d_{Q,H}(v)$, is given by $$d_{Q,H}(v)=\sum_{I\in E_{Q,H}(v)} \mathcal{N}(Q,H[I]).$$  The minimum $Q$-degree of $H$ is denoted by
$\delta_{Q}(H)$.

Let  $p\geq1$,  $Q$  be an $s$-vertex $r$-graph and $H$ be  an $n$-vertex $r$-graph, where $r\leq s\leq n$.
The \emph{Q-Lagrangian polynomial} $P_{Q,H}(\mathbf{x})$ of $H$ is defined as
\begin{displaymath}
\begin{split}
 P_{Q,H}(\mathbf{x})
 &=s!\sum_{\{i_{1},\ldots,i_{s}\}\in \binom{[n]}{s}}\mathcal{N}(Q,H[\{i_{1},\ldots,i_{s}\}])x_{i_{1}}\cdots x_{i_{s}}\\
 &=s!\sum_{\{i_{1},\ldots,i_{s}\}\in E(Q,H)}\mathcal{N}(Q,H[\{i_{1},\ldots,i_{s}\}])x_{i_{1}}\cdots x_{i_{s}},
\end{split}
\end{displaymath}
and the \emph{$(p,Q)$-spectral radius} $\lambda^{(p)}(Q,H)$ of $H$  is defined as
$$\lambda^{(p)}(Q,H)=\max_{\|\mathbf{x}\|_{p}=1}P_{Q,H}(\mathbf{x}),$$
where $\mathbf{x}=(x_{1},\ldots,x_{n})\in \mathbb{R}^{n}$  and $\|\mathbf{x}\|_{p}:=(|x_{1}|^{p}+\cdots +|x_{n}|^{p})^{1/p}$. 
It is noteworthy that the definition of the $(p,Q)$-spectral radius was recently introduced by Liu \cite{Liu2025}, and our definition here differs from Liu's by a constant factor $\frac{|Aut(Q)|}{s!}$.
When $Q=K_{s}^{r}$, we abbreviate $\lambda^{(p)}(Q,H)$ as $\lambda^{(p)}_{s}(H)$, termed the \emph{$s$-clique $p$-spectral radius}  of $H$.  If, further $Q=K_{r}^{r}$, we simple write $\lambda^{(p)}(H)$, recovering the $p$-spectral radius of $H$ introduced by Keevash, Lenz, and Mubayi\cite{KLM2014}.
If $\mathbf{x}\in \mathbb{R}^{n}$ is a vector such that $\|\mathbf{x}\|_{p}=1$ and
$\lambda^{(p)}(Q,H)=P_{Q,H}(\mathbf{x})$, then $\mathbf{x}$ is called a \emph{$Q$-eigenvector} of $H$ corresponding to $\lambda^{(p)}(Q,H)$. Clearly, there always exists a nonnegative $Q$-eigenvector corresponding to $\lambda^{(p)}(Q,H)$, called  a \emph{principal $Q$-eigenvector} of $H$.
Moreover, if a principal $Q$-eigenvector $\mathbf{x}$ is strictly  positive (i.e., $x_{v}>0$ for all $v\in V(H)$), then we call it a \emph{Perron-Frobenius $Q$-eigenvector} of $H$.

A property of $r$-graphs is a family of $r$-graphs closed under isomorphisms.  For a property $\mathcal{P}$, denoted by $\mathcal{P}_{n}$  the collection of $r$-graphs in $\mathcal{P}$ of order $n$. A property is called  \emph{hereditary} if it is closed under taking induced subgraphs. Given a  family  $\mathcal{F}$ of $r$-graphs, the class of all $\mathcal{F}$-free $r$-graphs forms a hereditary property, denoted by $\overline{\mathcal{F}}$. Throughout our discussion, we always assume that the hereditary property of 
$r$-graphs is closed under taking the disjoint union with an isolated vertex.
Given two $r$-graphs $Q$ and $H$, a map $\phi$: $V(Q)\rightarrow V(H)$ is a \emph{homomorphism} from $Q$ to $H$ if $\phi(e)\in E(H)$ for all $e\in E(Q)$. We say $Q$ is \emph{$H$-colorable} if there is a homomorphism from $Q$ to $H$.

A fundamental problem  in  extremal combinatorics can be formulated as follows: Given an $s$-vertex $r$-graph $Q$ and a hereditary property $\mathcal{P}$ of $r$-graphs, determine the extremal function
$$ex(Q,\mathcal{P}_{n}):=\max_{H\in \mathcal{P}_{n}}\mathcal{N}(Q,H).$$
By  Katona-Nemetz-Simonovits averaging argument \cite{KNS1964}, the ratio $ex(Q,\mathcal{P}_{n})/\binom{n}{s}$ is decreasing in $n$,
and so the limit

\begin{equation*}\label{eltt}
\pi(Q,\mathcal{P}):=\lim\limits_{n\to \infty}\frac{ex(Q,\mathcal{P}_{n})}{\binom{n}{s}}
\end{equation*}
always exists, called the \emph{$Q$-density} of $\mathcal{P}$. If $\mathcal{P}=\overline{\mathcal{F}}$  for
a family  $\mathcal{F}$ of $r$-graphs, then $ex(K_{r}^{r},\mathcal{P}_{n})$ and $\pi(K_{r}^{r},\mathcal{P})$ are  the \emph{Tur\'an number} and \emph{Tur\'an density} of $\mathcal{F}$, respectively. To maintain consistency in notation, we will use $ex(Q,\overline{\mathcal{F}}_{n})$ instead of $ex(n,Q,\mathcal{F})$ in the remaining part.

Similarly, we can study the spectral analogue of the aforementioned problem. For an $s$-vertex $r$-graph $Q$ and a hereditary property $\mathcal{P}$ of $r$-graphs, we define
$$\lambda^{(p)}(Q,\mathcal{P}_{n}):=\max_{H\in \mathcal{P}_{n}}\lambda^{(p)}(Q,H),$$
and the \emph{$(p,Q)$-spectral density} of $\mathcal{P}$ is defined as
$$\lambda^{(p)}(Q,\mathcal{P}):=\lim\limits_{n\to \infty}\frac{\lambda^{(p)}(Q,\mathcal{P}_{n})}{n^{s-s/p}}.$$

In \cite{N2014A}, Nikiforov conducted a systematic study of the  $p$-spectral radius  of hypergraphs using analytical methods, and proved that $\pi(K_{r}^{r},\mathcal{P})=\lambda^{(p)}(K_{r}^{r},\mathcal{P})$ holds for any $p>1$ and any  hereditary property $\mathcal{P}$ of $r$-graphs. Liu and Bu \cite{LB2023} introduced the $s$-clique spectral radius of a graph $G$(equivalent to $\lambda^{(s)}_{s}(G)$), and extended the spectral Mantel's theorem via the clique tensor. Yu and  Peng \cite{YP2025} gave a spectral version of the generalized Erd\H{o}s-Gallai theorem via 
the clique tensor. In \cite{Liu2025}, Liu  established a  general theorem that extends the result of Keevash-Lenz-Mubayi  and  obtained a spectral Erd\H{o}s pentagon theorem.

In this paper, we investigate spectral extremal problems concerning the $(p,Q)$-spectral radius of hypergraphs. For any hereditary property $\mathcal{P}$ of $r$-graphs, we prove that the \emph{$(p,Q)$-spectral density} of $\mathcal{P}$ exists for all $p\geq1$. Moreover, we show that the $Q$-density of $\mathcal{P}$ coincides with its  $(p,Q)$-spectral density when $p>1$. Furthermore, we study spectral generalized Tur\'an problems. In particular, we establish a spectral stability result: if  the maximum $(p,Q)$-spectral radius among all $\mathcal{F}$-free $r$-graphs satisfies a specific growth condition, then the extremal hypergraphs must possess a large minimum $Q$-degree. As an application, we derive a spectral analogue of the Erd\H{o}s Pentagon Problem. Specifically, for any $p\geq1$ and all sufficiently large $n$, the balanced 
blow-up of $C_{5}$ attains the maximal $(p,C_{5})$-spectral radius  over all $n$-vertex triangle-free graphs. This  extends the result of Liu \cite{Liu2025}. Additionally, we demonstrate  that for $p\geq1$ and $n$ sufficiently large, the $l$-partite Tur\'an graph $T_{l}(n)$ achieves the maximum $s$-clique $p$-spectral radius among all $n$-vertex $F$-free graphs, where $F$ is an edge-critical graph with $\chi(F)=l+1$. This  establishes a spectral counterpart to  the result of Ma and Qiu \cite{MQ2020}  and extends a theorem of Yu and Peng \cite{YP2025}.

\section{Preliminaries}
%In this paper, we raise the following problem:
%\begin{pro}
%For  $p\geq 1$,  an $r$-graph $Q$, and  a family $\mathcal{F}$ of $r$-%graphs, determine
%the maximum  of $\lambda^{(p)}(Q,H)$ over all $n$-vertex $\mathcal{F}$-free $r$-graphs $H$ .
%\end{pro}following the research trajectory of \cite{N2014A}, 
In this section, we present some fundamental properties of the parameter
$\lambda^{(p)}(Q,H)$. Hereafter, when given an $s$-vertex $r$-graph $Q$ and  an $n$-vertex $r$-graph $H$, it is always assumed that $n\geq s\geq r\geq 2$, provided no ambiguity arises.
\begin{prop}\label{conti}
 Let $Q$  be an $s$-vertex $r$-graph and $H$ be an $n$-vertex $r$-graph. If $p\geq1$, then $\lambda^{(p)}(Q,H)$ is an increasing and continuous function in $p$.
Moreover,
$$\lim\limits_{p\to \infty}\lambda^{(p)}(Q,H)=s!\mathcal{N}(Q,H).$$
\end{prop}
\begin{proof}
Since $\lambda^{(p)}(Q,H)$  always has a nonnegative $Q$-eigenvector, we obtain the following equivalent definition of  $\lambda^{(p)}(Q,H)$:
\begin{equation}\label{edef}
\lambda^{(p)}(Q,H)=\max_{|x_{1}|+\cdots+|x_{n}|=1}s!\sum_{\{i_{1},\ldots,i_{s}\}\in E(Q,H)}\mathcal{N}(Q,H[\{i_{1},\ldots,i_{s}\}])|x_{i_{1}}|^{1/p}\cdots |x_{i_{s}}|^{1/p},
\end{equation}
where $\mathbf{x}=(x_{1},\ldots,x_{n})\in \mathbb{R}^{n}$.
Note that $0\leq|x_{i_{1}}|\cdots |x_{i_{s}}|\leq1$. We now claim that for any $b\geq a\geq1$,
$$0\leq|x_{i_{1}}|^{1/b}\cdots |x_{i_{s}}|^{1/b}-|x_{i_{1}}|^{1/a}\cdots |x_{i_{s}}|^{1/a}\leq b-a.$$
The left inequality holds trivially,  and the right equality  holds when $|x_{i_{1}}|\cdots |x_{i_{s}}|=0 ~~\mbox{or}~~1$.
For $0<|x_{i_{1}}|\cdots |x_{i_{s}}|<1$, we apply the Mean Value Theorem  to the function $f(x):=(|x_{i_{1}}|\cdots |x_{i_{s}}|)^{1/x}$. There exists  $\xi\in (a,b)$ such that
\begin{displaymath}
\begin{split}
|x_{i_{1}}|^{1/b}\cdots |x_{i_{s}}|^{1/b}-|x_{i_{1}}|^{1/a}\cdots |x_{i_{s}}|^{1/a}
&=(b-a)\xi^{-2}(|x_{i_{1}}|\cdots |x_{i_{s}}|)^{\xi^{-1}}\ln(|x_{i_{1}}|\cdots |x_{i_{s}}|)^{-1}\\
&\leq (b-a)(|x_{i_{1}}|\cdots |x_{i_{s}}|)^{\xi^{-1}-\xi^{-2}}\\
&\leq b-a.
\end{split}
\end{displaymath}
Thus, the claim holds.

Let $\mathbf{y}=(y_1,\ldots,y_n)$ be a nonnegative vector such that  equality (\ref{edef}) holds for $\lambda^{(a)}(Q,H)$. Then,
$$\lambda^{(b)}(Q,H)-\lambda^{(a)}(Q,H)\geq s!\sum_{\{i_{1},\ldots,i_{s}\}\in E(Q,H)}\mathcal{N}(Q,H[\{i_{1},\ldots,i_{s}\}])(y_{i_{1}}^{1/b}\cdots y_{i_{s}}^{1/b}-y_{i_{1}}^{1/a}\cdots y_{i_{s}}^{1/a})\geq 0.$$
This implies that $\lambda^{(p)}(Q,H)$ is increasing in $p$.

Now, let $\mathbf{z}=(z_1,\ldots,z_n)$ be a nonnegative vector such that equality (\ref{edef}) holds for $\lambda^{(b)}(Q,H)$.
Then,
\begin{displaymath}
\begin{split}
0\leq\lambda^{(b)}(Q,H)-\lambda^{(a)}(Q,H)&\leq
s!\sum_{\{i_{1},\ldots,i_{s}\}\in E(Q,H)}\mathcal{N}(Q,H[\{i_{1},\ldots,i_{s}\}])(z_{i_{1}}^{1/b}\cdots z_{i_{s}}^{1/b}-z_{i_{1}}^{1/a}\cdots z_{i_{s}}^{1/a})\\
&\leq (b-a)s!\mathcal{N}(Q,H).
\end{split}
\end{displaymath}
Therefore,  $\lambda^{(p)}(Q,H)$ satisfies the Lipschitz condition and  is thus continuous.

By the definition of $\lambda^{(p)}(Q,H)$, it is evident that $\lambda^{(p)}(Q,H)\leq s!\mathcal{N}(Q,H)$.
On the other hand, taking the $n$-vector $\mathbf{x}=(n^{-1/p},\ldots,n^{-1/p})$ yields
$$\lambda^{(p)}(Q,H)\geq P_{Q,H}(\mathbf{x})=s!\mathcal{N}(Q,H)/n^{s/p}.$$
Thus, we obtain $$s!\mathcal{N}(Q,H)/n^{s/p}\leq \lambda^{(p)}(Q,H)\leq s!\mathcal{N}(Q,H),$$
which implies $\lim\limits_{p\to \infty}\lambda^{(p)}(Q,H)=s!\mathcal{N}(Q,H)$. 
This completes the proof.
\end{proof}

For a vertex subset $U\subseteq V(H)$ of an $n$-vertex $r$-graph $H$, we write $x_{U}=\Pi_{v\in U}x_{v}$. For $p>1$, the principal Q-eigenvector $\mathbf{x}=(x_{1},\ldots,x_{n})$ of  $H$ satisfies the following system eigenequations derived from Lagrange's method:
\begin{equation}\label{eigen}
\lambda^{(p)}(Q,H)x_{i}^{p-1}=(s-1)!\sum_{I\in E_{Q,H}(i)}\mathcal{N}(Q,H[I])x_{I\backslash\{i\}},~ \mbox{for}~
1\leq i\leq n.
\end{equation}

\begin{lem}\label{incr}
Let $p\geq1$, and let $Q$ be an $s$-vertex  $r$-graph and $H$ be an $n$-vertex $r$-graph. Then function
$$f_{Q,H}(p)=\bigg(\frac{\lambda^{(p)}(Q,H)}{s!\mathcal{N}(Q,H)}\bigg)^{p}$$
is decreasing in $p$.
\end{lem}
\begin{proof}
Set $\beta\geq\alpha\geq1$ and $\mathcal{N}:=\mathcal{N}(Q,H)$.  Let $\mathbf{x}=(x_{1},\ldots,x_{n})$ be a principal $Q$-eigenvector corresponding to $\lambda^{(\beta)}(Q,H)$. Using Power-Mean inequality, we obtain
$$\frac{\lambda^{(\beta)}(Q,H)}{s!\mathcal{N}}=\frac{1}{\mathcal{N}}
\sum_{I\in E(Q,H)}\mathcal{N}(Q,H[I])x_{I}
\leq\bigg(\frac{1}{\mathcal{N}}
\sum_{I\in E(Q,H)}\mathcal{N}(Q,H[I])(x_{I})^{\beta/\alpha}\bigg)^{\alpha/\beta}.$$
Note that $$\big(x_{1}^{\beta/\alpha}\big)^{\alpha}+\cdots+\big(x_{n}^{\beta/\alpha}\big)^{\alpha}=x_{1}^{\beta}+\cdots+x_{n}^{\beta}=1.$$
Thus, we have
$$\frac{1}{\mathcal{N}}
\sum_{I\in E(Q,H)}\mathcal{N}(Q,H[I])(x_{I})^{\beta/\alpha}\leq\frac{1}{s!\mathcal{N}}
\lambda^{(\alpha)}(Q,H),$$
and so $$\bigg(\frac{\lambda^{(\beta)}(Q,H)}{s!\mathcal{N}}\bigg)^{\beta}\leq \bigg(\frac{\lambda^{(\alpha)}(Q,H)}{s!\mathcal{N}}\bigg)^{\alpha},$$
completing the proof.
\end{proof}

We conclude this section with the following obvious result.
\begin{prop}
Let $p\geq1$, and let $Q$ be an $s$-vertex  $r$-graph and $H$ be an $n$-vertex $r$-graph. If $G$ is a subgraph of $H$, then
$\lambda^{(p)}(Q,G)\leq \lambda^{(p)}(Q,H).$ 
\end{prop}

\section{Extremal $(p,Q)$-spectral radius of hereditary families}
In this section,  we show that for any hereditary property  $\mathcal{P}$ of $r$-graphs, the $Q$-density of $\mathcal{P}$ is equal to its $(p,Q)$-spectral density when $p>1$, namely $\pi(Q,\mathcal{P})=\lambda^{(p)}(Q,\mathcal{P})$. We then investigate the $(p,Q)$-spectral radius of  hereditary families which satisfy  $\pi(Q,\mathcal{P})=\lambda^{(1)}(Q,\mathcal{P})$.

\begin{fact}[\hspace{1sp}\cite{ZLS2025}]\label{fact}
If $p>1$ and $s\geq2$, then the function
$$f(x)=\frac{1-sx}{(1-x)^{s/p}}$$
is decreasing for $0\leq x<1$.
\end{fact}

For a vector $\mathbf{x}\in \mathbb{R}^{n}$, we use the notation $\mathbf{x}_{\textup{min}}$ to represent the smallest element in the vector $\mathbf{x}$.

\begin{thm}\label{exi}
Let $p\geq 1$, and let $Q$ be an $r$-graph on $s$ vertices. If $\mathcal{P}$ is a hereditary property of $r$-graphs, then the limit
$$\lambda^{(p)}(Q,\mathcal{P})=\lim\limits_{n\to \infty}\lambda^{(p)}(Q,\mathcal{P}_{n})n^{s/p-s}$$
exists. If $p=1$, then $\lambda^{(1)}(Q,\mathcal{P}_{n})$ is increasing, and so
$$\lambda^{(1)}(Q,\mathcal{P}_{n})\leq \lambda^{(1)}(Q,\mathcal{P}).$$
If $p>1$, then $\lambda^{(p)}(Q,\mathcal{P})$ satisfies
$$\lambda^{(p)}(Q,\mathcal{P})\leq\frac{\lambda^{(p)}(Q,\mathcal{P}_{n})n^{s/p}}{(n)_{s}},$$
where $(n)_{s}=n(n-1)\cdots (n-s+1)$.
\end{thm}

\begin{proof}
%For every integer $n> s$, define $\lambda^{(p)}_{n}:=\lambda^{(p)}(Q,\mathcal{P}_{n})$.with $\|\mathbf{x}\|_{p}=1$ 
Let $H\in \mathcal{P}_{n}$
be an $r$-graph satisfying $\lambda^{(p)}(Q,H)=\lambda^{(p)}(Q,\mathcal{P}_{n})$, and let  $\mathbf{x}=(x_1,\ldots,x_n)$ be a principal $Q$-eigenvector corresponding to $\lambda^{(p)}(Q,H)$.
Through previous assumption about hereditary properties, we have 
$$\lambda^{(p)}(Q,\mathcal{P}_{n})\leq \lambda^{(p)}(Q,H+u)\leq \lambda^{(p)}(Q,\mathcal{P}_{n+1}),$$
where $u\notin V(H)$ and $H+u\in \mathcal{P}_{n+1}$ is an $r$-graph with vertex set $V(H+u)=V(H)\cup \{u\}$ and edge set $E(H+u)=E(H)$.  Thus, $\lambda^{(p)}(Q,\mathcal{P}_{n})$ is increasing  in $n$.

Recall that $\mathcal{N}(Q,K_{s}^{r})=\frac{s!}{|Aut(Q)|}\leq s!$. For $p=1$, by Maclaurin's inequality, we have
$$\lambda^{(1)}(Q,\mathcal{P}_{n})\leq s!\sum_{\{i_{1},\ldots,i_{s}\}\in \binom{[n]}{s}}s!x_{i_{1}}\cdots x_{i_{s}}\leq s!(x_{1}+\cdots+x_{n})^{s}= s!.$$
Thus, the sequence $\Big\{\lambda^{(1)}(Q,\mathcal{P}_{n})\Big\}_{n=1}^{\infty}$  converges to a limit $\lambda$,  and we conclude
$$\lambda=\lim\limits_{n\to \infty}\lambda^{(1)}(Q,\mathcal{P}_{n})n^{s-s}=\lambda^{(1)}(Q,\mathcal{P}).$$
For $p>1$, let $k\in V(H)$ be a vertex with $x_{k}=\mathbf{x}_{\textup{min}}$, and let $\mathbf{x}'$ be the $(n-1)$-vector obtained from $\mathbf{x}$ by removing the component $x_{k}$.  By (\ref{eigen}), we have
$$P_{Q,H-k}(\mathbf{x}')=\lambda^{(p)}(Q,H)-s!x_{k}\sum_{I\in E_{Q,H}(k)}\mathcal{N}(Q,H[I])x_{I\backslash\{k\}}=
\lambda^{(p)}(Q,\mathcal{P}_{n})-s\lambda^{(p)}(Q,\mathcal{P}_{n})x_{k}^{p}.$$
Since $\mathcal{P}$ is hereditary, $H-k\in \mathcal{P}_{n-1}$. Therefore,
$$\lambda^{(p)}(Q,\mathcal{P}_{n})(1-sx_{k}^{p})=P_{Q,H-k}(\mathbf{x}')\leq \lambda^{(p)}(Q,H-k)(\|\mathbf{x}' \|_{p}^{s})
\leq \lambda^{(p)}(Q,\mathcal{P}_{n-1})(1-x_{k}^{p})^{s/p},$$
or equivalently,
\begin{equation}\label{e8}
\frac{\lambda^{(p)}(Q,\mathcal{P}_{n-1})}{\lambda^{(p)}(Q,\mathcal{P}_{n})}\geq \frac{1-sx_{k}^{p}}{(1-x_{k}^{p})^{s/p}}.
\end{equation}
Note that $(\mathbf{x}_{\textup{min}})^{p}\leq 1/n$,
by (\ref{e8}) and Fact \ref{fact}, we have
\begin{equation*}
\frac{\lambda^{(p)}(Q,\mathcal{P}_{n-1})}{\lambda^{(p)}(Q,\mathcal{P}_{n})}
\geq\frac{1-s(\mathbf{x}_{\textup{min}})^{p}}{(1-(\mathbf{x}_{\textup{min}})^{p})^{s/p}}
\geq\frac{1-s/n}{(1-1/n)^{s/p}}.
\end{equation*}
This implies that
$$\frac{\lambda^{(p)}(Q,\mathcal{P}_{n-1})(n-1)^{s/p}}{(n-1)_{s}}\geq \frac{\lambda^{(p)}(Q,\mathcal{P}_{n})n^{s/p}}{(n)_{s}}.$$
Therefore, the sequence $\Big\{\frac{\lambda^{(p)}(Q,\mathcal{P}_{n})n^{s/p}}{(n)_{s}}\Big\}_{n=1}^{\infty}$ is decreasing and hence convergent. This completes the proof.
\end{proof}

\subsection{The equivalence of $\lambda^{(p)}(Q,\mathcal{P})$ and $\pi(Q,\mathcal{P})$}
Given a hereditary property $\mathcal{P}$ of $r$-graphs. For $H\in \mathcal{P}_{n}$ with $\mathcal{N}(Q,H)=ex(Q,\mathcal{P}_{n})$, the $n$-vector $\mathbf{x}=(n^{-1/p},\ldots,n^{-1/p})$ yields
\begin{equation}\label{zhy}
\lambda^{(p)}(Q,H)\geq P_{Q,H}(\mathbf{x})=s!\mathcal{N}(Q,H)/n^{s/p}
=s!ex(Q,\mathcal{P}_{n})/n^{s/p}.
\end{equation}
Thus
$$\lambda^{(p)}(Q,\mathcal{P}_{n})\geq \lambda^{(p)}(Q,H) \geq s!ex(Q,\mathcal{P}_{n})/n^{s/p},$$
which implies
$$\frac{\lambda^{(p)}(Q,\mathcal{P}_{n})n^{s/p}}{(n)_{s}}\geq \frac{ex(Q,\mathcal{P}_{n})}{\binom{n}{s}}.$$
Taking  $n\to \infty$ and applying Theorem \ref{exi},   we obtain for $p\geq1$,
\begin{equation}\label{pp}
\lambda^{(p)}(Q,\mathcal{P})\geq\pi(Q,\mathcal{P}).
 \end{equation}

We now state one of our main results: we show that for $p>1$, equality in inequality (\ref{pp}) always  holds. This significantly extends the result of Nikiforov \cite[Theorem 12]{N2014A}.
\begin{thm}\label{t1}
If $Q$ is an $r$-graph and $\mathcal{P}$ is a hereditary property of $r$-graphs, then for every $p>1$,
$$\lambda^{(p)}(Q,\mathcal{P})=\pi(Q,\mathcal{P}).$$
\end{thm}

In the following, we present several lemmas necessary for the proof of Theorem \ref{t1}.

\begin{lem}\label{t3}
Let $Q$ be an $r$-graph on $s$ vertices, and $\mathcal{P}$ be a hereditary property of $r$-graphs with $\lambda^{(p)}(Q,\mathcal{P})>0$. If $p>1$ and $\lambda_{n}^{(p)}:=
\lambda^{(p)}(Q,\mathcal{P}_{n})$, then there exist infinitely many $n$ such that
$$\frac{\lambda_{n-1}^{(p)}(n-1)^{s/p}}{(n-1)_{s}}-\frac{\lambda_{n}^{(p)}n^{s/p}}{(n)_{s}}<\frac{1}{n\log n}\cdot\frac{\lambda_{n}^{(p)}n^{s/p}}{(n)_{s}}.$$
\end{lem}
\begin{proof}
Assume for a contradiction that there exists $n_{0}$ such that for all $n\geq n_{0}$,
$$\frac{\lambda_{n-1}^{(p)}(n-1)^{s/p}}{(n-1)_{s}}-\frac{\lambda_{n}^{(p)}n^{s/p}}{(n)_{s}}\geq\frac{1}{n\log n}\cdot\frac{\lambda_{n}^{(p)}n^{s/p}}{(n)_{s}}.$$
Summing the inequalities for all $n_0,n_0+1,\ldots, k$, we get
\begin{eqnarray*}
\frac{\lambda_{n_{0}-1}^{(p)}(n_{0}-1)^{s/p}}{(n_{0}-1)_{s}}-\frac{\lambda_{k}^{(p)}k^{s/p}}{(k)_{s}}
&=& \sum_{n=n_{0}}^{k}
\left( \frac{\lambda_{n-1}^{(p)}(n-1)^{s/p}}{(n-1)_{s}}-\frac{\lambda_{n}^{(p)}n^{s/p}}{(n)_{s}} \right) \\
&\geq& \sum_{n=n_{0}}^{k} \frac{1}{n\log n}\cdot\frac{\lambda_{n}^{(p)}n^{s/p}}{(n)_{s}}\\
&\geq& \lambda^{(p)}(Q,\mathcal{P})\sum_{n=n_{0}}^{k} \frac{1}{n\log n},
\end{eqnarray*}
where the last inequality follows from Theorem \ref{exi}. Note that the left-hand side is bounded and the right-hand side
diverges. Taking $k$ sufficiently large, we obtain a contradiction.
\end{proof}

\begin{lem}\label{t4}
Let $Q$ be an $r$-graph on $s$ vertices, and $\mathcal{P}$ be a hereditary  property of $r$-graphs with $\lambda^{(p)}(Q,\mathcal{P})>0$. Suppose that $H_{n}\in \mathcal{P}_{n}$ is an $r$-graph satisfying $\lambda^{(p)}(Q,H_{n})=\lambda^{(p)}(Q,\mathcal{P}_{n})$ for $p>1$ and $\mathbf{x}=(x_1,\ldots,x_n)$ is a principal $Q$-eigenvector corresponding to $\lambda^{(p)}(Q,H_{n})$.
Then there exist infinitely many $n$ such that
$$(\mathbf{x}_{\textup{min}})^{p}\geq\frac{1}{n}\Big(1-\frac{p}{(p-1)s\log n}\Big).$$
\end{lem}
\begin{proof}
Assume, for contradiction,  that there exists $n_{0}$ such that for all $n>n_{0}$,
\begin{equation}\label{e6}
(\mathbf{x}_{\textup{min}})^{p}<\frac{1}{n}\Big(1-\frac{p}{(p-1)s\log n}\Big).
\end{equation}
Set  $\lambda_{n}^{(p)}:=\lambda^{(p)}(Q,\mathcal{P}_{n})$. By Lemma \ref{t3}, we can select  sufficiently large $n>n_{0}$ such that
$$\frac{\lambda_{n-1}^{(p)}(n-1)^{s/p}}{(n-1)_{s}}-\frac{\lambda_{n}^{(p)}n^{s/p}}{(n)_{s}}<\frac{1}{n\log n}\cdot\frac{\lambda_{n}^{(p)}n^{s/p}}{(n)_{s}},$$
and so
\begin{equation}\label{e7}
\frac{\lambda_{n-1}^{(p)}}{\lambda_{n}^{(p)}}<\frac{n^{s/p-1}(n-s)}{(n-1)^{s/p}}\Big(1+\frac{1}{n\log n}\Big).
\end{equation}
Let $k\in V(H_{n})$ be a vertex with $x_{k}=\mathbf{x}_{\textup{min}}$.
Then, by (\ref{e8}) and (\ref{e7}), we obtain
$$\frac{1-sx_{k}^{p}}{(1-x_{k}^{p})^{s/p}}\leq \frac{\lambda_{n-1}^{(p)}}{\lambda_{n}^{(p)}}
\leq \frac{n^{s/p-1}(n-s)}{(n-1)^{s/p}}\Big(1+\frac{1}{n\log n}\Big).$$
Applying Fact \ref{fact} and (\ref{e6}), we derive
$$\frac{1-\frac{s}{n}(1-\frac{p}{(p-1)s\log n})}{\Big({1-\frac{1}{n}(1-\frac{p}{(p-1)s\log n})}\Big)^{s/p}}
\leq \frac{1-sx_{k}^{p}}{(1-x_{k}^{p})^{s/p}}\leq\frac{n^{s/p-1}(n-s)}{(n-1)^{s/p}}\Big(1+\frac{1}{n\log n}\Big),$$
and hence
$$\frac{\Big(n-s+\frac{p}{(p-1)\log n}\Big)n^{s/p-1}}
{\Big(n-1+\frac{p}{(p-1)s\log n}\Big)^{s/p}}
\leq\frac{n^{s/p-1}(n-s)}{(n-1)^{s/p}}\Big(1+\frac{1}{n\log n}\Big).$$
This can be simplified  to
\begin{equation}\label{subs}
1+\frac{p}{(p-1)(n-s)\log n}\leq\Big(1+\frac{p}{(p-1)s(n-1)\log n}\Big)^{s/p}\Big(1+\frac{1}{n\log n}\Big).
\end{equation}
For sufficiently large $n$, we have
\begin{displaymath}
\begin{split}
\Big(1+\frac{p}{(p-1)s(n-1)\log n}\Big)^{s/p}
&=1+\frac{1}{(p-1)(n-1)\log n}+O\Big(\frac{1}{(n\log n)^{2}}\Big)\\
&\leq1+\frac{1}{(p-1)(n-1)\log n}+\frac{1}{(p-1)(n-1)(n-2)\log n}\\
&=1+\frac{1}{(p-1)(n-2)\log n}.
\end{split}
\end{displaymath}
Substituting this bound into (\ref{subs}), we obtain
$$1+\frac{p}{(p-1)(n-s)\log n}\leq \Big(1+\frac{1}{(p-1)(n-2)\log n}\Big)\Big(1+\frac{1}{n\log n}\Big).$$
By some cancellations and rearranging, we get
$$\frac{p}{n-s}\leq \frac{1}{n-2}+\frac{p-1}{n}+\frac{1}{n(n-2)\log n}.$$
Noting that $\frac{p}{n-s}\geq\frac{p}{n-2}$, we have
$$2(p-1)\leq\frac{1}{\log n},$$
which leads to a contradiction for sufficiently large $n$.
\end{proof}

\begin{lem}\label{t2}
Let  $p>1$, and let $Q$  be an $s$-vertex $r$-graph and $H$ be an $n$-vertex $r$-graph with $\lambda^{(p)}(Q,H)=\lambda$ and  minimum $Q$-degree $\delta$. Let $\mathbf{x}$ be a principal $Q$-eigenvector corresponding to  $\lambda$. Then
$$\Big(\frac{\lambda (\mathbf{x}_{\textup{min}})^{p-1}}{(s-1)!}\Big)^{p}
\leq \frac{s!\binom{n}{s-1}\delta^{p-1}}{n^{s-1}}-(s!\tbinom{n}{s-1}\delta^{p-1}-\delta^{p})(\mathbf{x}_{\textup{min}})^{p(s-1)}.$$
\end{lem}
\begin{proof}
Set  $V:=V(H)$, and let $k\in V$ be a vertex achieving the minimum $Q$-degree $\delta$. Considering the eigenequation for $\lambda^{(p)}(Q,H)$  at  vertex  $k$:
$$\lambda (\mathbf{x}_{\textup{min}})^{p-1}\leq \lambda x^{p-1}_{k}=(s-1)!\sum_{I\in E_{Q,H}(k)}\mathcal{N}(Q,H[I])x_{I\backslash\{k\}}.$$
 By H\"older's inequality,  we have
\begin{equation}\label{e1}
\bigg(\frac{\lambda (\mathbf{x}_{\textup{min}})^{p-1}}{(s-1)!}\bigg)^{p}\leq \delta^{p-1}\sum_{I\in E_{Q,H}(k)}\mathcal{N}(Q,H[I])(x_{I\backslash\{k\}})^{p}.
\end{equation}
Define $T_{1}=\{I_{1}\in \binom{V}{s-1}:I_{1}\cup \{k\} \in E_{Q,H}(k)\}$ and $T_{2}=\{I_{2}\in \binom{V}{s-1}:I_{2}\cup \{k\} \notin E_{Q,H}(k)\}$. Then 
\begin{align}\label{e2}
\begin{split}
\sum_{I\in E_{Q,H}(k)}\mathcal{N}(Q,H[I])(x_{I\backslash\{k\}})^{p}
=&\sum_{I\in \binom{V}{s-1}}s!x_{I}^{p}
-\sum_{I_{1}\in T_{1}}(s!-\mathcal{N}(Q,H[I_{1}\cup \{k\}]))x_{I_{1}}^{p}-\sum_{I_{2}\in T_{2}}s!x_{I_{2}}^{p}\\
\leq& \sum_{I\in \binom{V}{s-1}}s!x_{I}^{p}
-\sum_{I_{1}\in T_{1}}(s!-\mathcal{N}(Q,H[I_{1}\cup \{k\}]))(\mathbf{x}_{\textup{min}})^{p(s-1)}\\
&-\sum_{I_{2}\in T_{2}}s!(\mathbf{x}_{\textup{min}})^{p(s-1)}\\
=&\sum_{I\in \binom{V}{s-1}}s!x_{I}^{p}-\big(s!\tbinom{n}{s-1}-\delta\big)(\mathbf{x}_{\textup{min}})^{p(s-1)}.
\end{split}
\end{align}
By Maclaurin's inequality, we have
\begin{equation}\label{e3}
\sum_{I\in \binom{V}{s-1}}x_{I}^{p}\leq \binom{n}{s-1}\Bigg(n^{-1}\sum_{i\in V}x_{i}^{p}\Bigg)^{s-1}=
\frac{\binom{n}{s-1}}{n^{s-1}}.
\end{equation}
The result is obtained by combining the inequalities (\ref{e1}), (\ref{e2}), and (\ref{e3}).
\end{proof}

Now we present the proof of Theorem \ref{t1}. 

\begin{proof}[{\bf Proof of Theorem \ref{t1}}]
Observe that if $\lambda^{(p)}(Q,\mathcal{P})=0$, then it follows from  inequality (\ref{pp})  that $\pi(Q,\mathcal{P})=0$.  

Next,  assume  $\lambda^{(p)}(Q,\mathcal{P})>0$.
Suppose that $H_{n}\in \mathcal{P}_{n}$ is an $r$-graph with $\lambda^{(p)}(Q,H_{n})=\lambda^{(p)}(Q,\mathcal{P}_{n})=:\lambda$
and minimum $Q$-degree $\delta$. 
Let $\mathbf{x}=(x_1,\ldots,x_n)$ be a principal $Q$-eigenvector corresponding to $\lambda^{(p)}(Q,H_{n})$.
By Lemma \ref{t4},  there exists an increasing infinite  sequence   $\{n_{i}\}_{i=1}^{\infty}$ of positive integers such that
for each $n\in \{n_1,n_2,\ldots \}$, 
$$(\mathbf{x}_{\textup{min}})^{p}\geq\frac{1}{n}\Big(1-\frac{p}{(p-1)s\log n}\Big).$$

From Theorem  \ref{exi} and Lemma \ref{t2}, we derive
\begin{align*}
\begin{split}
\big(1-o(1)\big)\Big(\frac{\lambda^{(p)}(Q,\mathcal{P})}{(s-1)!}\Big)^{p}\cdot
\frac{((n)_{s})^{p}}{n^{s+p-1}}
&\leq\Big(\frac{\lambda (\mathbf{x}_{\textup{min}})^{p-1}}{(s-1)!}\Big)^{p}\\
&\leq \frac{s!\binom{n}{s-1}\delta^{p-1}}{n^{s-1}}-(s!\tbinom{n}{s-1}\delta^{p-1}-\delta^{p})(\mathbf{x}_{\textup{min}})^{p(s-1)}\\
&\leq \frac{s!\binom{n}{s-1}\delta^{p-1}}{n^{s-1}}-\frac{s!\binom{n}{s-1}\delta^{p-1}}{n^{s-1}}\big(1-o(1)\big)+\frac{\delta^{p}}{n^{s-1}}\\
&\leq o(n^{(s-1)(p-1)})+\frac{\delta^{p}}{n^{s-1}}.
\end{split}
\end{align*}

Since $\delta\leq \frac{s\mathcal{N}(Q,H_{n})}{n}\leq \frac{sex(Q,\mathcal{P}_{n})}{n}$, it follows that
$$(1-o(1))(\lambda^{(p)}(Q,\mathcal{P}))^{p}
\leq o(1)+\Big(\frac{ex(Q,\mathcal{P}_{n})}{\binom{n}{s}}\Big)^{p},$$
where the term $o(1)$ tends to $0$ as $n\to \infty$.
Consequently,
\[
(\lambda^{(p)}(Q,\mathcal{P}))^{p}=\lim_{i \to \infty}(1-o(1))(\lambda^{(p)}(Q,\mathcal{P}))^{p}\leq
\lim_{i \to \infty}o(1)+\Big(\frac{ex(Q,\mathcal{P}_{n_{i}})}{\binom{n_{i}}{s}}\Big)^{p}=(\pi(Q,\mathcal{P}))^{p}.
\]
Combining this inequality  with (\ref{pp}) completes the proof of Theorem \ref{t1}.
\end{proof}

\begin{rmk}
The assumption that the hereditary property $\mathcal{P}$ of $r$-graphs is closed under disjoint union with an isolated vertex is only used to ensure the monotonicity of $\lambda^{(1)}(Q,\mathcal{P}_n)$ in $n$. For the case $p>1$, the proof of Theorem \ref{t1} does not rely on this condition. Consequently, the theorem remains valid for any hereditary property $\mathcal{P}$ of $r$-graphs.
\end{rmk}

 The celebrated  Erd\H{o}s-Stone-Simonovits theorem states that
$$\mathrm{ex}(n,F)= \bigg(1-\frac{1}{\chi(F)-1}+o(1) \bigg)\frac{n^{2}}{2},$$ where $\chi(F)$ is the chromatic number of $F$. In \cite{AS2016}, Alon and Shikhelman extended the Erd\H{o}s-Stone-Simonovits theorem to count copies of $K_{s}$.

 \begin{lem}[\hspace{1sp}\cite{AS2016}]\label{eESS}
Let $F$ be a graph with $\chi(F)=k$. Then
$$ex(K_{s},\overline{F}_{n})=\binom{k-1}{s}\bigg(\frac{n}{k-1}\bigg)^{s}+o(n^{s}).$$
 \end{lem}

Applying Theorem \ref{t1} and Lemma \ref{eESS} yields  the following result directly.

\begin{cor}\label{sESS}
Let $p>1$ and $F$ be a graph with $\chi(F)=k$. Then
$$\lambda^{(p)}(K_{s},\overline{F}_{n})=\frac{(k-1)_{s}}{(k-1)^{s}}n^{s-s/p}+o(n^{s-s/p}).$$
 \end{cor}

\begin{rmk}
The case where $p=s=2$ in Corollary \ref{sESS} corresponds to the spectral  Erd\H{o}s-Stone-Simonovits theorem by Nikiforov \cite{N2009}. Thus, Corollary \ref{sESS} can be viewed as a  generalization  of adjacency spectral version of the Erd\H{o}s-Stone-Simonovits theorem.  
\end{rmk}

Gerbner and Palmer \cite{GP2019} gave a further extension  to count arbitrary graphs $H$ using the regularity lemma.
\begin{lem}[\hspace{1sp}\cite{GP2019}]\label{GP}
Let $H$ be an $s$-vertex graph, and let $F$ be a graph with $\chi(F)=k$. Then
$$ex(H,\overline{F}_{n})=ex(H, (\overline{K_{k}})_{n})+o(n^{s}).$$
\end{lem}

Below, we present the spectral version of Lemma \ref{GP}.

\begin{cor}
Let $p>1$, and let $H$ be an $s$-vertex graph and $F$ be a graph with $\chi(F)=k$. Then
$$\lambda^{(p)}(H,\overline{F}_{n})=\lambda^{(p)}(H, (\overline{ K_{k}})_{n})+o(n^{s-s/p}).$$
\end{cor}
\begin{proof}
By Theorem \ref{t1}, we have
$$\frac{\lambda^{(p)}(H,\overline{F}_{n})}{n^{s-s/p}}=\big(1+o(1)\big)\frac{ex(H,\overline{F}_{n})}{\binom{n}{s}},$$
and thus,
$$\lambda^{(p)}(H,\overline{F}_{n})=s!ex(H,\overline{F}_{n})n^{-s/p}+o(n^{s-s/p}).$$
Similarly, 
$$\lambda^{(p)}(H,(\overline{K_{k}})_{n})=s!ex(H,(\overline{K_{k}})_{n})n^{-s/p}+o(n^{s-s/p}).$$
By Lemma \ref{GP}, we obtain
\begin{align*}
\begin{split}
\lambda^{(p)}(H,\overline{F}_{n})-\lambda^{(p)}(H,(\overline{K_{k}})_{n})&=
s!(ex(H,\overline{F}_{n})-ex(H,(\overline{K_{k}})_{n}))n^{-s/p}+o(n^{s-s/p})\\
&=o(n^{s-s/p}),
\end{split}
\end{align*}
which completes the proof.
\end{proof}

\subsection{$Q$-flat properties of $r$-graphs}

Consider   an $r$-graph $H$ on $n$ vertices and a sequence of positive integers $k_{1},\ldots,k_{n}$. The \emph{blow-up} of $H$ with respect to $k_{1},\ldots,k_{n}$, denoted by $H(k_{1},\ldots,k_{n})$, is the $r$-graph obtained by replacing each vertex $i\in V(H)$ with a  vertex class   $V_{i}$ (also called a block) of size $k_{i}$,  and if $\{i_{1},\ldots,i_{r}\}\in E(H)$, then $\{i_{1,j_{1}},\ldots,i_{r,j_{r}}\} \in E(H(k_{1},\ldots,k_{n}))$ for every $i_{1,j_{1}}\in V_{i_{1}},\ldots,i_{r,j_{r}}\in V_{i_{r}}$.
A  property $\mathcal{M}$ of $r$-graphs is  \emph{multiplicative} if $H\in \mathcal{M}$  implies that every blow-up $H(k_{1},\ldots,k_{n})$ is also in $\mathcal{M}$ 
(i.e., $\mathcal{M}$  is closed under the blow-up operation).

For an  $s$-vertex $r$-graph $Q$ and
a hereditary property $\mathcal{P}$ of $r$-graphs,  we say that $\mathcal{P}$ is  \emph{Q-flat} if $\lambda^{(1)}(Q,\mathcal{P})=\pi(Q,\mathcal{P})$. We establish the following
sufficient condition for $Q$-flat properties.

 \begin{lem}\label{flat}
Let $Q$ be an  $r$-graph and $\mathcal{P}$ be a  hereditary and multiplicative property of $r$-graphs. Then $\mathcal{P}$  is $Q$-flat, and hence $\lambda^{(p)}(Q,\mathcal{P})=\pi(Q,\mathcal{P})$ for every $p\geq 1$.
\end{lem}
\begin{proof}
By Theorem \ref{t1}, we have $\lambda^{(p)}(Q,\mathcal{P})=\pi(Q,\mathcal{P})$ for all $p>1$. Inequality (\ref{pp}) gives $\lambda^{(1)}(Q,\mathcal{P})\geq\pi(Q,\mathcal{P})$. It remains to show the reverse inequality $\lambda^{(1)}(Q,\mathcal{P})\leq\pi(Q,\mathcal{P})$.

Consider $H\in \mathcal{P}_{n}$ with
$\lambda^{(1)}(Q,H)=\lambda^{(1)}(Q,\mathcal{P}_{n})$,
and let $\mathbf{x}=(x_1,\ldots,x_n)$ be a principal $Q$-eigenvector corresponding to $\lambda^{(1)}(Q,H)$ (with $\|\mathbf{x}\|_{1}=1$).  We claim that
\begin{equation}\label{1f}
\lambda^{(1)}(Q,H)=P_{Q,H}(\mathbf{x})\leq \pi(Q,\mathcal{P}).
\end{equation}
Since  $P_{Q,H}(\mathbf{x})$ is continuous in each variable, it  suffices to prove inequality (\ref{1f})  for positive rational numbers $x_{1},\ldots,x_{n}$. Thus, we can assume that
$$x_{1}=k_{1}/k,\ldots,x_{n}=k_{n}/k,$$
where $k,k_{1},\ldots,k_{n}$ are positive integers and $k=k_{1}+\cdots+k_{n}$. Consequently, inequality (\ref{1f}) is equivalent to
\begin{equation}\label{1f1f}
\frac{P_{Q,H}((k_{1},\ldots,k_{n}))}{k^{s}}\leq \pi(Q,\mathcal{P}).
\end{equation}

Let $H(k_{1},\ldots,k_{n})$ be the blow-up of $H$ with blocks $V_{1},\ldots,V_{n}$, and let $Q$ be an $s$-vertex $r$-graph.  For  any $i_{1},\ldots,i_{s}\in V(H)$, the
subgraph
$H(k_{1},\ldots,k_{n})[V_{i_{1}}\cup\ldots\cup V_{i_{s}}]$ contains at least  $k_{i_1}\times\cdots\times k_{i_s}$ copies of $H[\{i_{1},\ldots,i_{s}\}]$. It follows that
$$\mathcal{N}(Q,H(k_{1},\ldots,k_{n})[V_{i_{1}}\cup\ldots\cup V_{i_{s}}])\geq k_{i_{1}}\cdots k_{i_{s}}\mathcal{N}(Q,H[\{i_{1},\ldots,i_{s}\}]).$$
This implies that
\begin{align*}
\begin{split}
P_{Q,H}((k_{1},\ldots,k_{n}))&=s!\sum_{\{i_{1},\ldots,i_{s}\}\in E(Q,H)}\mathcal{N}(Q,H[\{i_{1},\ldots,i_{s}\}])k_{i_{1}}\cdots k_{i_{s}}\\
&\leq s!\sum_{\{i_{1},\ldots,i_{s}\}\in E(Q,H)}\mathcal{N}(Q,H(k_{1},\ldots,k_{n})[V_{i_{1}}\cup\ldots\cup V_{i_{s}}])\\
&\leq P_{Q,H(k_{1},\ldots,k_{n})}((1,\ldots,1))=s!\mathcal{N}(Q,H(k_{1},\ldots,k_{n})).
\end{split}
\end{align*}
Since $H(k_{1},\ldots,k_{n})\in\mathcal{P}$ (as $\mathcal{P}$ is multiplicative) and $\nu(H(k_{1},\ldots,k_{n}))=k$, we obtain
$$\frac{s!\mathcal{N}(Q,H(k_{1},\ldots,k_{n}))}{k^{s}}\leq \frac{ex(Q,\mathcal{P}_{k})}{\binom{k}{s}}
\leq \pi(Q,\mathcal{P})+o(1),$$
where the term $o(1)$ tends to $0$ as $k\to \infty$.

Similarly, for every positive integer $t$, we have
\begin{align*}
\begin{split}
\frac{P_{Q,H}((k_{1},\ldots,k_{n}))}{k^{s}}
&=\frac{P_{Q,H}((t k_{1},\ldots,t k_{n}))}{(t k)^{s}}\\
&\leq \frac{s!\mathcal{N}(Q,H(tk_{1},\ldots,tk_{n}))}{(t k)^{s}}\\
&\leq \pi(Q,\mathcal{P})+o(1).
\end{split}
\end{align*}
Taking  $t\to \infty$, we establish inequality (\ref{1f1f}), hence inequality (\ref{1f}) holds. Therefore,
 $$\lambda^{(1)}(Q,\mathcal{P})=\lim\limits_{n\to \infty}\lambda^{(1)}(Q,\mathcal{P}_{n})\leq\pi(Q,\mathcal{P}),$$
  completing the proof.
\end{proof}

Notably, the complete graph $K_{l+1}$ is $2$-covering, which implies that $\overline{K_{l+1}}$ is  $K_{2}$-flat.
A classical theorem of Tur\'an \cite[p. $294$]{B1978} establishes that for any $K_{l+1}$-free graph $G$ on $n$ vertices, the number of edges satisfies $e(G)\leq (1-\frac{1}{l})\frac{n^{2}}{2}$.
Wilf \cite{W1986} later provided a spectral extension of Tur\'an's theorem, demonstrating that if $G$ is an $n$-vertex $K_{l+1}$-free graph, then its largest eigenvalue $\lambda(G)$ satisfies
$\lambda(G)\leq(1-\frac{1}{l})n$. In 2002, Nikiforov \cite{N2002} further extended this result by proving that for any  $K_{l+1}$-free graph $G$ with $m$ edges, 
$\lambda(G)\leq(1-\frac{1}{l})^{1/2}(2m)^{1/2}$.  In the following, we  generalize these bounds to  families of $r$-graphs with $Q$-flat properties.

 %there exists  concise and tight upper
 %bounds on $e(G)$ and $\lambda^{(\alpha)}(G)$ for every $G\in \mathcal{H}$.

 \begin{thm}\label{t8}
If $Q$ is an $s$-vertex $r$-graph and $\mathcal{P}$ is a $Q$-flat property of $r$-graphs, then for any $H\in \mathcal{P}_{n}$,
\begin{equation*}
\mathcal{N} (Q,H)\leq\pi(Q,\mathcal{P})n^{s}/s!,
\end{equation*}
and for every $p\geq1$,
\begin{equation*}
\lambda^{(p)}(Q,H)\leq\pi(Q,\mathcal{P})n^{s-s/p}.
\end{equation*}
\end{thm}
\begin{proof}
Let $\mathbf{x}=(x_1,\ldots,x_n)$ be a $Q$-principal eigenvector corresponding to $\lambda^{(p)}(Q,H)$ with $\|\mathbf{x}\|_{p}=1$. Then
$$\frac{\lambda^{(p)}(Q,H)}{n^{s-s/p}}= \frac{P_{Q,H}(\mathbf{x})}{n^{s-s/p}}
=P_{Q,H}((x_1/n^{1-1/p},\ldots,x_n/n^{1-1/p})).$$
By Power-Mean inequality, for any $p\geq1$,
$$\frac{x_1+\cdots+x_n}{n^{1-1/p}}\leq(x_1^p+\cdots+x_n^p)^{1/p}=1.$$
From  Theorem \ref{exi}, it follows that
$$\frac{\lambda^{(p)}(Q,H)}{n^{s-s/p}}\leq \lambda^{(1)}(Q,H)\leq \lambda^{(1)}(Q,\mathcal{P}_{n})\leq\lambda^{(1)}(Q,\mathcal{P})=\pi(Q,\mathcal{P}).$$
Since $\lambda^{(p)}(Q,H)\geq s!\mathcal{N}(Q,H)/n^{s/p}$, we conclude
$$\mathcal{N}(Q,H)\leq\pi(Q,\mathcal{P})n^{s}/s!,$$
completing the proof.
\end{proof}

\begin{lem}\label{ttt6}
 If $Q$ is an $s$-vertex $r$-graph and  $\mathcal{P}$ is a $Q$-flat property of $r$-graphs, then for any $p\geq1$ and $H\in \mathcal{P}$,
$$\lambda^{(p)}(Q,H)\leq\pi(Q,\mathcal{P})^{1/p}(s!\mathcal{N}(Q,H))^{1-1/p}.$$
\end{lem}
\begin{proof}
Lemma \ref{incr} implies that  $$\big(\lambda^{(p)}(Q,H)\big)^{p}\leq\lambda^{(1)}(Q,H)(s!\mathcal{N}(Q,H))^{p-1}.$$
Moreover, since $\mathcal{P}$ is  $Q$-flat,  Theorem \ref{exi} yields
$$\lambda^{(1)}(Q,H)\leq\lambda^{(1)}(Q,\mathcal{P})=\pi(Q,\mathcal{P}).$$
Combining these inequalities, we obtain 
$$\lambda^{(p)}(Q,H)\leq\pi(Q,\mathcal{P})^{1/p}(s!\mathcal{N}(Q,H))^{1-1/p},$$
completing the proof.
\end{proof}

\section{Spectral generalized   Tur\'an problems}
In this section, we study  spectral generalized   Tur\'an problems for a family of $\mathcal{F}$-free  $r$-graphs. We  assume that all members of $\mathcal{F}$ contain no isolated vertices.

The following spectral stability theorem indicates that if  the maximum $(p,Q)$-spectral radius over all $\mathcal{F}$-free $r$-graphs satisfies a specific growth condition, then the extremal hypergraphs must have a large minimum $Q$-degree.
\begin{thm}\label{cri}
Let $p>1$, $s\geq r\geq2$, and $0<\varepsilon<1$. Let $Q$ be an $s$-vertex $r$-graph, and let $\mathcal{F}$ be a family of $r$-graphs with  $\pi(Q,\overline{\mathcal{F}})>0$.
Let $\mathcal{G}_n$ be the collection  of all $n$-vertex $\mathcal{F}$-free $r$-graphs with minimum $Q$-degree at least $(1-\varepsilon)\pi(Q,\overline{\mathcal{F}})\binom{n}{s-1}$ and define $\lambda^{(p)}(Q,\mathcal{G}_{n})=\max\{\lambda^{(p)}(Q,G): G\in \mathcal{G}_{n}\}$. Suppose that there exists a sufficiently large $n_{0}\in \mathbb{N}$  such that for every $n\geq n_{0}$, we have
\begin{equation}\label{d12d}
\lambda^{(p)}(Q,\overline{\mathcal{F}}_{n}) \ge  \lambda^{(p)}(Q,\overline{\mathcal{F}}_{n-1}) + \pi(Q,\overline{\mathcal{F}})(s-s/p) (1-\sigma)n^{s-s/p-1},
\end{equation}
where $\sigma=\varepsilon \pi(Q,\overline{\mathcal{F}})/(5s!(s-1)).$
Then  for any $\mathcal{F}$-free $r$-graph $H$ on $n\geq n_{0}$ vertices, we have
$$\lambda^{(p)}(Q,H)\leq \lambda^{(p)}(Q,\mathcal{G}_{n}).$$
In addition, if the equality holds, then $H\in\mathcal{G}_{n}$.
\end{thm}

We require the following Lemma for the proof of Theorem \ref{cri}.

\begin{lem}\label{tt2}
Let $p>1$, $s\geq r\geq2$, and $0<\varepsilon<1$.
Let  $Q$ be an $s$-vertex $r$-graph, and let $\mathcal{P}$ be a hereditary  property of $r$-graphs with $\pi(Q,\mathcal{P})>0$. Let $H_{n}\in \mathcal{P}_{n}$ satisfy $\lambda^{(p)}(Q,H_{n})=\lambda^{(p)}(Q,\mathcal{P}_{n})$. Suppose $0\leq\varepsilon'<\varepsilon \pi(Q,\mathcal{P})/(s!(s-1))$, and let $\mathbf{x}$ be a principal $Q$-eigenvector corresponding to  $\lambda^{(p)}(Q,H_{n})$. If
$n$ is sufficiently large and
$(\mathbf{x}_{\textup{min}})^{p}\geq\frac{1-\varepsilon'}{n},$
then
$$\delta_{Q}(H_{n})\geq(1-\varepsilon)\pi(Q,\mathcal{P})\binom{n}{s-1}.$$
\end{lem}
\begin{proof}
 Set $\delta:=\delta_{Q}(H_{n})$ and $\lambda:=\lambda^{(p)}(Q,H_{n})$.
 Suppose for  contradiction that $\delta<(1-\varepsilon)\pi(Q,\mathcal{P})\binom{n}{s-1}$. By Theorem \ref{exi} and Lemma \ref{t2}, we obtain
 \begin{align*}
\begin{split}
(1-\varepsilon')^{p-1}\Big(\frac{\lambda^{(p)}(Q,\mathcal{P})}{(s-1)!}\Big)^{p}\cdot
\frac{((n)_{s})^{p}}{n^{s+p-1}}
& \leq\Big(\frac{\lambda (\mathbf{x}_{\textup{min}})^{p-1}}{(s-1)!}\Big)^{p}\\
&\leq \frac{s!\binom{n}{s-1}\delta^{p-1}}{n^{s-1}}-(s!\tbinom{n}{s-1}\delta^{p-1}-\delta^{p})(\mathbf{x}_{\textup{min}})^{p(s-1)}\\
&\leq  \frac{s!\binom{n}{s-1}\delta^{p-1}}{n^{s-1}}-\frac{1}{n^{s-1}}(s!\tbinom{n}{s-1}\delta^{p-1}-\delta^{p})(1-(s-1)\varepsilon')\\
&\leq \frac{s!(s-1)\varepsilon'\binom{n}{s-1}\delta^{p-1}}{n^{s-1}}+\frac{\delta^{p}}{n^{s-1}}\\
&\leq \frac{\binom{n}{s-1}\delta^{p-1}}{n^{s-1}}(s!(s-1)\varepsilon'+
(1-\varepsilon)\pi(Q,\mathcal{P}))\\
&\leq \frac{\pi(Q,\mathcal{P})\binom{n}{s-1}\delta^{p-1}}{n^{s-1}},
\end{split}
\end{align*}
where the last inequality follows from $\varepsilon'<\varepsilon \pi(Q,\mathcal{P})/(s!(s-1))$.
By Theorem \ref{t1} and the assumption $\delta<(1-\varepsilon)\pi(Q,\mathcal{P})\binom{n}{s-1}$, we further get
$$(1-\varepsilon')^{p-1}\big(\pi(Q,\mathcal{P})\tbinom{n-1}{s-1}\big)^{p}\leq(1-\varepsilon)^{p-1}\big(\pi(Q,\mathcal{P})\tbinom{n}{s-1}\big)^{p},$$
which implies that
$$1<\Big(\frac{1-\varepsilon'}{1-\varepsilon}\Big)^{\frac{p-1}{p}}\leq \frac{n}{n-s+1}.$$
This is a contradiction for sufficiently large $n$, completing the proof.
\end{proof}

\begin{proof}[{\bf Proof of Theorem \ref{cri}}]
Let $H_{n}$ be an $\mathcal{F}$-free $r$-graph on $n$ vertices that satisfies $\lambda^{(p)}(Q,H_{n})=\lambda^{(p)}(Q,\overline{\mathcal{F}}_{n})$ , and let $\mathbf{x}=(x_{1},\ldots,x_{n})$ be a principal  $Q$-eigenvector corresponding to $\lambda^{(p)}(Q,H_{n})$. In view of Lemma \ref{tt2}, it suffices to show that for $n\geq n_{0}$, $$(\mathbf{x}_{\textup{min}})^{p}\geq\frac{1-\varepsilon'}{n},$$
where $\varepsilon'=\varepsilon \pi(Q,\overline{\mathcal{F}})/(2s!(s-1))$.
Suppose for contradiction  that for some $n$,
$$(\mathbf{x}_{\textup{min}})^{p}<\frac{1-\varepsilon'}{n}.$$
Applying (\ref{e8}), Fact \ref{fact}, and Bernoulli's inequality, we obtain
\begin{align*}
\begin{split}
\frac{\lambda^{(p)}(Q,H_{n-1})}{\lambda^{(p)}(Q,H_{n})}
&\geq\frac{1-s(\mathbf{x}_{\textup{min}})^{p}}{(1-(\mathbf{x}_{\textup{min}})^{p})^{s/p}}\\
&\geq\Big(1-\frac{s(1-\varepsilon')}{n}\Big)\Big(1-\frac{1-\varepsilon'}{n}\Big)^{-s/p}\\
&\geq\Big(1-\frac{s(1-\varepsilon')}{n}\Big)\Big(1+\frac{s(1-\varepsilon')}{pn}\Big)\\
&= 1-\frac{(s-s/p)(1-\varepsilon')}{n}-\frac{s^{2}(1-\varepsilon')^{2}}{pn^{2}}.
\end{split}
\end{align*}
From Theorem \ref{t1}, it follows that $\lambda^{(p)}(Q,H_{n})=
(\pi(Q,\overline{\mathcal{F}})+o(1))n^{s-s/p}$, and hence
\begin{equation}\label{k1}
\lambda^{(p)}(Q,H_{n-1})\geq \lambda^{(p)}(Q,H_{n})-\pi(Q,\overline{\mathcal{F}})(s-s/p)(1-\varepsilon'/2)n^{s-s/p-1}.
\end{equation}

On the other hand, by (\ref{d12d}), we have
\begin{equation}\label{k2}
\lambda^{(p)}(Q,H_{n})\geq \lambda^{(p)}(Q,H_{n-1})+\pi(Q,\overline{\mathcal{F}})(s-s/p)(1-\sigma)n^{s-s/p-1}.
\end{equation}
Combining (\ref{k1}) and (\ref{k2}) yields
$$\pi(Q,\overline{\mathcal{F}})(s-s/p)(1-\sigma)n^{s-s/p-1}\leq \pi(Q,\overline{\mathcal{F}})(s-s/p)(1-\varepsilon'/2)n^{s-s/p-1},$$
which  contradicts $\sigma=\varepsilon \pi(Q,\overline{\mathcal{F}})/(5s!(s-1))$.
This completes the  proof of Theorem \ref{cri}.
\end{proof}

%\begin{defi}[$Q$-degree stability]
%Let $s\geq r\geq2$, $Q$ be an $s$-vertex $r$-graph and $\mathcal{F}$ %be a  family of $r$-graphs. Let $\mathfrak{H}$ be a family of %$\mathcal{F}$-free $r$-graphs. We say $\mathcal{F}$ is \emph{$Q$-%degree-stable} with respect  to $\mathfrak{H}$  if  there exist %constants $n_{0}$ and $\varepsilon>0$ such that  every $\mathcal{F}$-%free $r$-graph $\mathcal{H}$ on $n\geq n_{0}$ vertices with %$\delta_{Q}(\mathcal{H})\geq (\pi(Q, \overline{\mathcal{F}})/(s-1)!-%\varepsilon)n^{s-1}$ is a member of $\mathfrak{H}$.
%\end{defi}

\subsection{Spectral Erd\H{o}s pentagon theorem} In $1984$, Erd\H{o}s
 conjectured that for every $n\geq5$, the balanced blow-up of $C_{5}$ contains the maximum number of copies of $C_{5}$ among all $n$-vertex triangle-free graphs. This conjecture was first resolved independently  by Grzesik \cite{G2012} and Hatami et al. \cite{HHK2013} for sufficiently large $n$. Later, Lidick\'y and Pfender \cite{LP2018} completed the proof by extending the result to all $n$.
\begin{lem}[\hspace{1sp}\cite{LP2018}]\label{BF}
For all $n$, the maximum number of copies of  $C_{5}$ in $K_{3}$-free graphs on $n$ vertices is $$\prod^{4}_{i=0}\Big\lfloor\frac{n+i}{5}\Big\rfloor.$$
Moreover, for $n\geq9$, the  only $K_{3}$-free graph on $n$ vertices maximizing the number of copies of  $C_{5}$ is the balanced blow-up of $C_{5}$.
\end{lem}

\begin{lem}[\hspace{1sp}\cite{CL2024}]\label{c5d}
There exist $\varepsilon>0$ and $N_{0}$ such that the following holds for all $n\geq N_{0}$. If $G$ is an $n$-vertex $K_{3}$-free graph with  $\delta_{C_{5}}(G)\geq (1/5^{4}-\varepsilon)n^{4}$, then $G$
is $C_{5}$-colorable.
\end{lem}

\begin{rmk}
 We remark that the minimum $Q$-degree in \cite{CL2024} differs from our by a constant  factor of $|Aut(Q)|$. Additionally, observe that $|Aut(C_{5})|=10$.
 \end{rmk}

\begin{lem}[\hspace{1sp}\cite{ZLF2024}]\label{t9}
Let $l\geq r\geq2$. Then 
$e(T^{r}_{l}(n))=\frac{(l)_{r}}{r!l^{r}}n^{r}+O(n^{r-2}).$ 
\end{lem}

\begin{lem}[\hspace{1sp}\cite{KNY2014}]\label{kpt}
Let $l\geq r\geq2$, and let $G$ be an $l$-partite $r$-graph of order $n$.
For every $p>1$,
$$\lambda^{(p)}(G)\leq \lambda^{(p)}(T^{r}_{l}(n)),$$
with equality if and only if $G=T^{r}_{l}(n)$.
\end{lem}

For any  $r$-graph $Q$ on $s$ vertices and any $r$-graph $H$, we define  $D(Q,H)$ as the $s$-graph derived from $H$ with vertex set $V(D(Q,H))=V(H)$  and edge set $$E(D(Q,H))=\{\{v_{1},\ldots,v_{s}\}:H[v_{1},\ldots,v_{s}]\supseteq Q\}.$$
Note that if $\mathcal{N}(Q,H[v_{1},\ldots,v_{s}])=1$ for any $\{v_{1},\ldots,v_{s}\}\in E(D(Q,H))$, then 
\begin{equation}\label{t10}
\mathcal{N}(Q,H)=e(D(Q,H))~~\mbox{and}~~\lambda^{(p)}(Q,H)=\lambda^{(p)}(D(Q,H)).
\end{equation}

Recently, Liu \cite[Theorem $1.5$]{Liu2025} established a  general theorem that extends the result of Keevash-Lenz-Mubayi  and applied it to obtain a spectral Erd\H{o}s pentagon theorem. We extend  Liu's result via a different  approach. For any $p\geq1$, an $r$-graph $Q$ and a family $\mathcal{G}_{n}$ of $r$-graphs on $n$ vertices, let $\lambda^{(p)}(Q,\mathcal{G}_{n})$ (resp. $\lambda^{(p)}(\mathcal{G}_{n})$)  denote the maximum $(p,Q)$-spectral radius (resp. $p$-spectral radius) among all $r$-graphs in $\mathcal{G}_{n}$.

\begin{thm}
Let $p\geq 1$, and let $\mathcal{L}_{n}$ be the balanced blow-up of $C_{5}$ on $n$ vertices. Then, for all  sufficiently large $n$ and any $n$-vertex $K_{3}$-free graph $G$, we have $\lambda^{(p)}(C_{5},G)\leq
\lambda^{(p)}(C_{5},\mathcal{L}_{n})$. The equality holds  if and only if  $G=\mathcal{L}_{n}$ for $p>1$, and if  $ G\supseteq C_{5}$ for $p=1$.
\end{thm}
\begin{proof}
For any $C_{5}$-colorable graph $H$ with a homomorphism $\phi$ from $V(H)$ to $V(C_{5})$, denote $V(C_{5})$ as $\{1,2,3,4,5\}$. For each $i\in[5]$, define  $V_{i}=\{v\in V(H):  \phi(v)=i\}$ (some $V_{i}$ may be empty). This defines a  natural partition of $V(H)$. 

We claim that for any five vertices $v_{1},\ldots,v_{5}$ in $V(H)$, if $H[v_{1},\ldots,v_{5}]$ contains a copy of $C_{5}$ (in fact, $H[v_{1},\ldots,v_{5}]\cong C_{5}$), then these  five vertices must belong to five distinct parts in the partition.  

Suppose $v_{1}v_{2}v_{3}v_{4}v_{5}v_{1}$ forms a copy of $C_{5}$
in $H[v_{1},\ldots,v_{5}]$, with $\{v_{i}v_{i+1}\}\in E(H)$ for $i\in [5]$ (indices modulo $5$). Assume for contradiction that the claim fails.  By  symmetry, we may assume that $v_{1}$ and $v_{3}$ belong
to the same part, i.e., $\phi(v_{1})=\phi(v_{3})$. Since $H$ is 
$C_{5}$-colorable and $\{v_{5}v_{1}\} \in E(H)$, it follows that 
$$\{\{\phi(v_{5})\phi(v_{3})\},\{\phi(v_{3})\phi(v_{4})\},
\{\phi(v_{4})\phi(v_{5})\}\}\subseteq E(C_{5}).$$ 
which contradicts the fact that $C_{5}$ is $K_{3}$-free. Therefore, the claim holds. 

The above claim implies that $D(C_{5},H)$
is a $5$-partite $5$-graph. It follows that $D(C_{5},\mathcal{L}_{n})$ is isomorphic to $T_{5}^{5}(n)$, and hence $\mathcal{N}(C_{5},\mathcal{L}_{n})=e(T_{5}^{5}(n))$. By Lemmas \ref{BF} and \ref{t9}, we have
\begin{equation}\label{exe}
 ex(C_{5},(\overline{K_{3}})_{n})=\mathcal{N}(C_{5},\mathcal{L}_{n})=e(T_{5}^{5}(n))=\frac{n^{5}}{5^{5}}+O(n^{3}),
\end{equation}
and hence $\pi(C_{5},\overline{K_{3}})= 5!/5^{5}$.

Let $\mathrm{Col}(C_{5})_{n}$ be the set of all $C_{5}$-colorable graphs on $n$ vertices, and let $$\mathcal{R}_{n}:=\{D(C_{5},H): H\in \mathrm{Col}(C_{5})_{n}\}.$$
Then, by (\ref{t10}),
\begin{equation}\label{t11}
\lambda^{(p)}(C_{5},\mathrm{Col}(C_{5})_{n})=\lambda^{(p)}(\mathcal{R}_{n}).
\end{equation}
Lemma \ref{c5d} shows that there exist $\varepsilon>0$ and $N_{0}$ such that  for every $n$-vertex $K_{3}$-free graph $G$ with  $\delta_{C_{5}}(G)\geq(1/5^{4}-\varepsilon)n^{4}$ is contained in $\mathrm{Col}(C_{5})_{n}$.

By  (\ref{zhy}) and (\ref{exe}), we have
\begin{equation}\label{1k3}
\lambda^{(p)}(C_{5},(\overline{K_{3}})_{n})\geq 5!ex(C_{5},(\overline{K_{3}})_{n})/n^{5/p}\geq \pi(C_{5},\overline{K_{3}})n^{5-5/p}+O(n^{3-5/p}).
\end{equation}
Note that $K_{3}$ is a $2$-covering graph. Lemma \ref{flat} and Theorem \ref{t8} imply that
\begin{align}\label{2k3}
\begin{split}
\lambda^{(p)}(C_{5},(\overline{K_{3}})_{n-1})
&\leq \pi(C_{5},\overline{K_{3}})(n-1)^{5-5/p}\\
&=\pi(C_{5},\overline{K_{3}})n^{5-5/p}-
\pi(C_{5},\overline{K_{3}})(5-5/p)n^{4-5/p}+O(n^{3-5/p}).
\end{split}
\end{align}
Combining (\ref{1k3}) and (\ref{2k3}) yields
$$\lambda^{(p)}(C_{5},(\overline{K_{3}})_{n})-\lambda^{(p)}(C_{5},(\overline{K_{3}})_{n-1})
\geq \pi(C_{5},\overline{K_{3}})(5-5/p)n^{4-5/p}+o(n^{4-5/p}).$$
Thus, by Theorem \ref{cri} and equality (\ref{t11}), for $p>1$ and enough large $n$, we have
$$\lambda^{(p)}(C_{5},(\overline{K_{3}})_{n})\leq \lambda^{(p)}(C_{5},\mathrm{Col}(C_{5})_{n})=\lambda^{(p)}(\mathcal{R}_{n})\leq\lambda^{(p)}(T_{5}^{5}(n))=
\lambda^{(p)}(C_{5},\mathcal{L}_{n}),$$
where the third inequality follows from Lemma \ref{kpt}.

For $p=1$,
by Theorem \ref{t8}, we have
$\lambda^{(1)}(C_{5},(\overline{K_{3}})_{n})\leq\pi(C_{5},\overline{K_{3}})= 5!/5^{5}$. Moreover, observe that
 $$\lambda^{(1)}(C_{5},G)\geq\lambda^{(1)}(C_{5},C_{5})= \lambda^{(1)}(K^{5}_{5})=5!/5^{5},$$ which implies $\lambda^{(1)}(C_{5},G)=\lambda^{(1)}(C_{5},(\overline{K_{3}})_{n})$, completing the proof.
\end{proof}

\subsection{Spectral generalized Tur\'an problems for edge-critical graphs} For a graph $H$ and $e\in E(H)$, let $H-e$  denote the graph with vertex set $V(H)$ and edge set $E(H)\backslash\{e\}$.  A graph $H$ is called \emph{edge-critical} if  there exists an edge $e$ of $H$ such that $\chi(H-e)=\chi(H)-1$. The $s$-\emph{expansion} $H^{(s)}$  of $H$  is the $s$-graph obtained from $H$ by enlarging each edge of $H$ with $s-2$ new vertices disjoint from $V(H)$ such that distinct edges of $H$ are enlarged by distinct vertices.

Simonovits \cite{S1968} extended Tur\'an's theorem to any edge-critical graph $F$ and established the critical edge theorem. Later, Ma and Qiu \cite{MQ2020} generalized Simonovit's result as follows:

\begin{thm}[\hspace{1sp}\cite{MQ2020}]\label{mq}
Let  $l\geq s\geq 2$ , and let $F$ be an edge-critical graph with $\chi(F)=l+1$. Then for sufficiently large $n$,  the  unique $n$-vertex $F$-free graph with the maximum number of copies of $K_{s}$ is
the Tur\'an graph $T_{l}(n)$.
\end{thm}

Very recently, Zheng, Li and Su \cite{ZLS2025} determined the  maximum $p$-spectral radius  among all $n$-vertex $F^{(r)}$-free $s$-graphs, where  $F$ is an edge-critical graph.

\begin{lem}[\hspace{1sp}\cite{ZLS2025}]\label{ccg}
Let $p\geq1$, $l\geq s\geq 2$ , and let $F$ be an edge-critical graph with $\chi(F)=l+1$. Then there exists $n_{0}$, such that for any $F^{(s)}$-free $s$-graph $G$ on $n>n_{0}$ vertices, $\lambda^{(p)}(G)\leq\lambda^{(p)}(T^{s}_{l}(n))$. The equality holds  if and only if  $G=T^{s}_{l}(n)$ for $p>1$, and if  $ G\supseteq K^{s}_{l}$ for $p=1$.
\end{lem}

Let $H$ be an $r$-graph.  The  \emph{$2$-shadow} of $H$, denoted by $\partial_{2}H$, is the graph with vertex set $V(\partial_{2}H)=V(H)$ and edge set $E(\partial_{2}H)=\{\{v_{1},v_{2}\}: \{v_{1},v_{2}\}\subseteq e\in E(H)\}$. 

We present a spectral analogue of Theorem \ref{mq}.

\begin{thm}\label{fin}
Let $p\geq1$, $l\geq s\geq 2$ , and let $F$ be an edge-critical graph with $\chi(F)=l+1$. Then
there exists $n_{0}$, such that for any $F$-free graph $G$ on $n>n_{0}$ vertices, $\lambda^{(p)}_{s}(G)\leq\lambda^{(p)}_{s}(T_{l}(n))$. The equality holds  if and only if  $G=T_{l}(n)$ for $p>1$, and if  $ G\supseteq K_{l}$ for $p=1$.
\end{thm}
\begin{proof}
We define the following sets for a given integer $n$:
$$\mathcal{A}_{n}:=\{D(K_{s},G):G ~~\mbox{is}~~F\mbox{-free~~ graph~~on}~~n~~\mbox{vertices}\},$$
$$\mathcal{B}_{n}:=\{H:H~~\mbox{is~~an}~~s\mbox{-graph~~on}~~n~~\mbox{vertices~~and}~~\partial_{2}H~~\mbox{is}~~F\mbox{-free}\},$$
$$\mathcal{C}_{n}:=\{H:H~~\mbox{is~~an}~~s\mbox{-graph~~on}~~n~~\mbox{vertices~~and}~~H~~\mbox{is}~~F^{(s)}\mbox{-free}\}.$$
Observe that for any $F$-free graph $G$, we have $\partial_{2}D(K_{s},G)\subseteq G$, which implies $\mathcal{A}_{n}\subseteq \mathcal{B}_{n}\subseteq \mathcal{C}_{n}$. By Lemma \ref{ccg}, for $p\geq1$, it follows that
$$\lambda^{(p)}(\mathcal{A}_{n})\leq\lambda^{(p)}(\mathcal{C}_{n})=\lambda^{(p)}(T^{s}_{l}(n)).$$
Note that $D(K_{s},T_{l}(n))=T^{s}_{l}(n)$. From equality (\ref{t10}), we obtain
 $$\lambda^{(p)}(K_{s},\overline{F}_{n})=\lambda^{(p)}(\mathcal{A}_{n})=\lambda^{(p)}_{s}(T_{l}(n)).$$ The result follows from  Lemma \ref{ccg}.
\end{proof}
\begin{rmk}
From Theorem \ref{fin}, letting $p\to \infty$ and applying Proposition \ref{conti} directly yields Theorem \ref{mq}, but does not yield the uniqueness of the extremal graph. Moreover, Theorem \ref{fin} can be viewed as a generalization of  the result of Yu and Peng \cite[Theorem $10$]{YP2025}.
\end{rmk}

\section{Concluding remarks}
In this paper, we systematically investigate  the $(p,Q)$-spectral radius of hypergraphs and derive several results concerning spectral generalized  Tur\'an problems. Specifically, Theorem \ref{cri} establishes a spectral stability result. We conjecture that the conclusion of Theorem \ref{cri} holds even without condition (\ref{d12d}), leading to the following:
\begin{conj}\label{conj}
Let $p>1$, $s\geq r\geq2$, $Q$ be an $s$-vertex $r$-graph and $\mathcal{F}$ be a family of $r$-graphs with  $\pi(Q,\overline{\mathcal{F}})>0$.
Let $\mathcal{G}_n$
be the collection  of all $n$-vertex $\mathcal{F}$-free $r$-graphs with minimum $Q$-degree more than $(1-\varepsilon)\pi(Q,\overline{\mathcal{F}})\binom{n}{s-1}$ and $\lambda^{(p)}(Q,\mathcal{G}_{n})=\max\{\lambda^{(p)}(Q,G): G\in \mathcal{G}_{n}\}$.
Then  for any $\mathcal{F}$-free graph $H$ on $n\geq n_{0}$ vertices, we have
$$\lambda^{(p)}(Q,H)\leq \lambda^{(p)}(Q,\mathcal{G}_{n}).$$
In addition, if the equality holds, then $H\in\mathcal{G}_{n}$.
\end{conj}

To address this conjecture, we propose two potential approaches.

\begin{pro}\label{p1}
Let $Q$ be an $r$-graph on $s$ vertices, and $\mathcal{P}$ be a hereditary  property of $r$-graphs with $\pi(Q,\mathcal{P})>0$. Suppose that $H_{n}\in \mathcal{P}_{n}$ is an $r$-graph satisfying $\lambda^{(p)}(Q,H_{n})=\lambda^{(p)}(Q,\mathcal{P}_{n})$ for $p>1$ and $\mathbf{x}=(x_1,\ldots,x_n)$ is a principal $Q$-eigenvector for $\lambda^{(p)}(Q,H_{n})$. Does there exist a constant  $n_{0}$ such that
for all $n\geq n_{0}$,
$$(\mathbf{x}_{\textup{min}})^{p}\geq\frac{1}{n}\Big(1-\frac{p}{(p-1)s\log n}\Big)?$$
An affirmative answer to Problem \ref{p1} would, via Lemma \ref{tt2},  imply conjecture \ref{conj}.
\end{pro}

\begin{pro}\label{p2}
Let $Q$ be an $s$-vertex $r$-graph, and let $\mathcal{F}$ be a family of $r$-graphs with $\pi(Q,\overline{\mathcal{F}})>0$.  For $p>1$, does there exist a sequence $\{a_{n}\}$  such that
$$\lambda^{(p)}(Q,\overline{\mathcal{F}}_{n})=\pi(Q,\overline{\mathcal{F}})n^{s-s/p}+a_{n}n^{(s-1)(1-1/p)},$$
and the limit $\lim\limits_{n\to \infty} a_{n}$ exists? If answered affirmatively, then
$$\lambda^{(p)}(Q,\overline{\mathcal{F}}_{n})=\lambda^{(p)}(Q,\overline{\mathcal{F}}_{n-1}) + \pi(Q,\overline{\mathcal{F}})(s-s/p) n^{s-s/p-1}+o(n^{s-s/p-1}),$$
and conjecture \ref{conj} would follow from Theorem \ref{cri}.
\end{pro}

 Moreover,  Problem \ref{p2} is of independent interest. Generally speaking, to what structural parameters of the graph is $\{a_{n}\}$ related and for which hereditary families does the limit of $\{a_{n}\}$ exist?

For any  $r$-graph $Q$ on $s$ vertices and any $r$-graph $H$, recall  the definition of  $s$-graph $D(Q,H)$: its edge set is defined as
$$E(D(Q,H))=\{\{v_{1},\ldots,v_{s}\}:H[v_{1},\ldots,v_{s}]\supseteq Q\}.$$
Assigning a weight $\mathcal{N}(Q,H[v_{1},\ldots,v_{s}])$ to each
edge $\{v_{1},\ldots,v_{s}\}$, we define its $p$-spectral radius as:
$$\lambda^{(p)}(D(Q,H))=\max_{\|\mathbf{x}\|_{p}=1}s!\sum_{\{i_{1},\ldots,i_{s}\}\in E(D(Q,H))}\mathcal{N}(Q,H[\{i_{1},\ldots,i_{s}\}])x_{i_{1}}\cdots x_{i_{s}}.$$
Then, it follows that
\begin{equation*}
\lambda^{(p)}(Q,H)=\lambda^{(p)}(D(Q,H)).
\end{equation*}
This establishes a connection between the $(p,Q)$-spectral radius of the $r$-graph $H$ and the $p$-spectral radius of the weighted $s$-graph $D(Q,H)$. For relevant conclusions regarding the $p$-spectral radius of weighted hypergraphs, one may refer to the results of Nikiforov \cite{N2014B}, such as the Perron-Frobenius theory for the weighted hypergraphs discussed in Section $5$ of  \cite{N2014B}.

\section*{Data availability}
 No data was used for the research described in the article.

\end{document}